\documentclass[11pt,reqno,a4paper]{amsart}
\usepackage{amsmath}
\usepackage{amssymb}
\usepackage{graphicx}
\usepackage{enumerate}
\newcommand{\eps}{\varepsilon}
\newcommand{\N}{\mathbb{N}}

\newcommand{\R}{\mathbb{R}}

\DeclareMathOperator{\diam}{diam}

\title{Recurrence rates for observations of flows}
\author{J\'er\^ome Rousseau}
\address{Universidade Federal de Alagoas, Instituto de Matem\'atica, 57072-900 Macei\'o, AL, Brasil}
\email{jerome.rousseau@univ-brest.fr}
\urladdr{http://pageperso.univ-brest.fr/~rousseau}
\date{\today}

\begin{document}
\newtheorem{lem}{Lemma}[section]
\newtheorem{lemma}[lem]{Lemma}
\newtheorem{theorem}[lem]{Theorem}
\newtheorem{defi}[lem]{Definition}
\newtheorem{prop}[lem]{Proposition}
\newtheorem{corollary}[lem]{Corollary}
\newtheorem{example}[lem]{Example}
\newtheorem*{remark}{Remark}
\newtheorem{definition}[lem]{Definition} 
\newtheorem{proposition}[lem]{Proposition}  

\maketitle

\begin{abstract}
We study Poincar\'e recurrence for flows and observations of flows. For Anosov flow, we prove that the recurrence rates are linked to the local dimension of the invariant measure. More generally, we give for the recurrence rates for the observation an upper bound depending on the push-forward measure. When the flow is metrically isomorphic to a suspension flow for which the dynamic on the base is rapidly mixing, we prove the existence of a lower bound for the recurrence rates for the observation. We apply these results to the geodesic flow and we compute the recurrence rates for a particular observation of the geodesic flow, i.e. the projection on the manifold.
\end{abstract}

\section{Introduction}
\subsection{Poincar\'e Recurrence}
One of the fundamental theorem at the origin of dynamical system and ergodic theory is the Poincar\'e Recurrence Theorem. It states that, in a dynamical system, almost every orbit returns as closely as you wish to the initial point. The time needed for a point to come back, called the return time, has been widely studied in the last few years for discrete dynamical systems (e.g. \cite{MR1211492,MR1833809,MR1997964,MR2152396,MR1835750,MR1996808}).

A noteworthy work on the quantitative study of Poincar\'e recurrence is the article of Boshernitzan \cite{MR1231839}. In this work, Boshernitzan throw the idea of studying Poincar\'e recurrence for observations of dynamical systems. Indeed, he proved that for a dynamical system $(X,T,\mu)$ and an observation $f$ from $X$ to a metric space $(Y,d)$, whenever the $\alpha$-dimensional Hausdorff measure is $\sigma$-finite on $Y$ we have
\begin{equation}\label{bosh}\liminf_{n\rightarrow\infty}n^{1/\alpha}d\left(f(x),f(T^nx)\right)<\infty\qquad\textrm{for $\mu$-almost every $x$}.\end{equation}

Following this idea and the work of \cite{MR1833809,MR2191396}, recurrence rates for the observation  have been linked to the pointwise dimensions of the push-forward measure \cite{prfo}. An example of application of observations of dynamical systems is the study of recurrence for random dynamical systems \cite{RFRDS}.

It is natural to wonder if these results can be extend to continuous time, if one can obtain quantitative results of recurrence for flows and more generally for observations of flows. Barreira and Saussol \cite{MR1833809} proved that for a suspension flow over an Anosov diffeomorphism and such that the invariant measure is an equilibrium state of an H\"older potential, the return time of $\nu$-almost every point $y$ in the ball $B(y,r)$ behaves like $r^{-\dim_H\nu+1}$ when $r$ goes to zero (similar results have been proved for hitting time of Lorenz like flows \cite{galatolopacifico}).

P\`ene and Saussol \cite{MR2566164} studied the billiard flow in the plane with periodic configuration of scatterers, they proved that, almost everywhere, the return time of a point $(p,v)$ in the ball $B((p,v),r)$ is of the order $\exp(\frac{1}{r^2})$ and that the return time of the position of a point in $B(p,r)$, i.e. the return time of the projection of the flow on the billiard, is of the order $\exp(\frac{1}{r})$ almost everywhere.

Following these works and the idea of Boshernitzan we study the recurrence rates for flows and observations of flows, in particular we obtain some results for the geodesic flow.

\subsection{Statement of the principal results}
Let $M$ a compact Riemannian manifold and $d$ its induced metric. Let $\Psi$ a flow on $M$. Let $\nu$ a probability measure on $M$ invariant for the flow $\Psi$. We introduce the notion of return time and recurrence rates for flows: 
\begin{definition}We define for $x\in M$ the return time of the flow $\Psi$:
\[\tau_{r}^{\Psi}(x):=\inf\{t>\eta_r(x)\,:\,\Psi_t(x)\in B(x,r)\}\]
where $B(x,r)$ is the ball centered in $x$ and of radius $r$ and $\eta_r(x)$ is the first escape time of the ball $B(x,r)$, i.e. $\eta_r(x)=\inf\{t>0,\Psi_tx\notin B(x,r)\}$.
We define also the lower and upper recurrence rates:
\[\underline{R}^{\Psi}(x):=\liminf_{r\rightarrow0}\frac{\log\tau_{r}^{\Psi}(x)}{-\log r}\qquad\textrm{and}\qquad\overline{R}^{\Psi}(x):=\limsup_{r\rightarrow0}\frac{\log\tau_{r}^{\Psi}(x)}{-\log r}.\]
\end{definition}
We will show that these recurrence rates are linked to the local dimension of the invariant measure. We recall that \emph{the lower and upper pointwise or local dimension} of a Borel probability measure $\mu$ on a metric space $X$ at a point $x\in X$ are defined by
\[\underline{d}_\mu(x)=\underset{r\rightarrow0}{\liminf}\frac{\log\mu\left(B\left(x,r\right)\right)}{\log r}\qquad\textrm{and}\qquad\overline{d}_\mu(x)=\underset{r\rightarrow0}{\limsup}\frac{\log\mu\left(B\left(x,r\right)\right)}{\log r}.\]
 Firstly, we will prove a theorem satisfied for any flow:
\begin{theorem}\label{artoosftheseintroid}
Let $\Psi$ a differentiable flow on $M$ and $\nu$ an invariant probability measure for $\Psi$. For $\nu$-almost every $x\in M$ which is not a fixed point
\[\underline{R}^{\Psi}(x)\leq \underline{d}_\nu(x)-1\qquad\textrm{and}\qquad\overline{R}^{\Psi}(x)\leq \overline{d}_\nu(x)-1.\]
\end{theorem}
To obtain an equality between recurrence rates and dimensions, we need more assumptions on the system:
\begin{theorem}\label{artoosftheseintroideg}
Let $\Psi$ an Anosov flow on $M$. If $\nu$ is an equilibrium state of an H\"older potential, then 
\[\underline{R}^{\Psi}(x)= \underline{d}_\nu(x)-1\qquad\textrm{and}\qquad\overline{R}^{\Psi}(x)= \overline{d}_\nu(x)-1\]
for $\nu$-almost every $x\in M$.
\end{theorem}
The existence of the local dimension of the invariant measure for hyperbolic flows is still an open question. Indeed, even if for hyperbolic diffeomorphisms the existence of the local dimension has been proved \cite{MR819557,MR1709302} and there exists an explicit formula to compute it \cite{MR2237462}, for hyperbolic flows the existence has been proved only in the conformal case \cite{MR1814848,MR2257426}.

We can apply, for example, the previous theorem to the geodesic flow on a smooth manifold with striclty negative curvature. Since the geodesic flow is defined on the unit tangent bundle $T^1M$, we can also considered a particular observation of this flow: the position on the manifold $M$. Let $\Pi$ the canonical projection: 
\begin{eqnarray*}
\Pi\,\,:T^1M&\longrightarrow& M\\
(p,v)&\longmapsto& p.
\end{eqnarray*}
We study the return time for the canonical projection on the manifold $M$:
\[\tau_{r}^{\Psi,\Pi}(p,v):=\inf\{t>r\,:\,\Pi(\Psi_t(p,v))\in B(p,r)\}.\] 
Since $\Psi$ is the geodesic flow on $T^1M$, the first escape time of the projection of the flow on the manifold of the ball $B(p,r)$ is equal to $r$ for $r$ small enough. We define the recurrence rates for the canonical projection:
\[\underline{R}^{\Psi,\Pi}(p,v):=\liminf_{r\rightarrow0}\frac{\log\tau_{r}^{\Psi,\Pi}(p,v)}{-\log r}\qquad\textrm{and}\qquad\overline{R}^{\Psi,\Pi}(p,v):=\limsup_{r\rightarrow0}\frac{\log\tau_{r}^{\Psi,\Pi}(p,v)}{-\log r}.\]
\begin{theorem}\label{artoosfthesethgeogenintro}Let $\Psi$ the geodesic flow defined on $T^1M$ and $\nu$ an invariant probability measure for $\Psi$. Then for $\nu$-almost every $(p,v)\in T^1M$
\[\underline{R}^{\Psi,\Pi}(p,v)\leq\underline{d}_{\Pi_*\nu}(p)-1\qquad and \qquad \overline{R}^{\Psi,\Pi}(p,v)\leq\overline{d}_{\Pi_*\nu}(p)-1.\]
Moreover, if $M$ has a strictly negative curvature and if $\nu$ is an equilibrium state of an H\"older potential then
\[\underline{R}^{\Psi,\Pi}(p,v)=\underline{d}_{\Pi_*\nu}(p)-1\qquad and \qquad
\overline{R}^{\Psi,\Pi}=\overline{d}_{\Pi_*\nu}(p)-1\]
for $\nu$-almost every $(p,v)\in T^1M$ non-multiple such that $\underline{d}_{\Pi_*\nu}(p)>1$.
\end{theorem}
Since the geodesic flow preserves the Lebesgue measure on $T^1M$, we can apply Theorem~\ref{artoosftheseintroideg} and Theorem~\ref{artoosfthesethgeogenintro} to obtain the following noteworthy result:
\begin{corollary} Let $M$ a $n$-dimensional manifold with strictly negative curvature. Let $\Psi$ the geodesic flow defined on $T^1M$. Then for almost every $(p,v)\in T^1M$
\[{R}^{\Psi}(p,v)=2n-2\]
and
\[{R}^{\Psi,\Pi}(p,v)=n-1.\]
\end{corollary}
The structure of the paper is as follows. In a first time, we introduce the notion of escape function and give an upper bound for the recurrence rate for the observation in a general setting.

Then, in Section \ref{artoosfthesesecsusp} , using suspension flows, we obtain a lower bound of the recurrence rate for the observation for flows presenting some hyperbolic behaviour.

In Section \ref{artoosfthesesecgeo} ,we used these results to compute the recurrence rates of the geodesic flow and for a particular observation of the geodesic flow, the projection on the manifold.

Then, we give, in Section \ref{artoosfthesesecexemple} , some simple but essential examples of observations to understand the concepts of escape function and projection dimension.

Finally, in the last two sections, we prove the principal theorems.

\section{Recurrence for observations of flows}
Let $M$ a Riemannian manifold and $d$ its induced metric. Let $\mathcal{B}(M)$ the Borel $\sigma$-algebra of $M$ and let $\nu$ a probability measure on $M$. Let $\Psi=\left\{\Psi_t\right\}_{t\in\R}$ be a measurable flow on $M$. We recall that the measure $\nu$ is invariant for $\Psi$ if
\[\nu(\Psi_t^{-1}(A))=\nu(A)\qquad \forall t\in \R, \forall A\in \mathcal{B}(M).\]
Let $N\in\N$. As in \cite{prfo}, we are interested in observations of dynamical systems, we consider an observation $f:M\rightarrow \R^N$ and we study the image by $f$ of the orbits of $\Psi$. In particular, we are going to study Poincar\'e recurrence for the observation of the flow $\Psi$. First of all, we need to introduce the notion of {\it escape function}:
\begin{definition}\label{artoosftheseacceptable}
A function $\rho$ is an escape function if for all $1>\xi_1>0$, for all $\xi_2>1$, for all $\eps>0$, for $f_\star\nu$-almost every $z_1\in\R^N$, it exists $\zeta>0$ such that if $r<\zeta$, then for $f_\star\nu$-almost every $z_2\in B(z_1, \min\{(1-\xi_1)r,(\xi_2-1)r\})$ we have $r^{\eps}\rho_{\xi_1r}(z_2)\leq\rho_r(z_1)\leq\rho_{\xi_2 r}(z_2)r^{-\eps}$.
\end{definition}
To understand this concept, we give some examples of escape functions and we also refer to Section~\ref{artoosfthesesecexemple}:
\begin{enumerate}
\item For $\alpha\in[0,1]$, $$\rho_r(z)=r^\alpha|\log r|.$$
\item $$\rho_r(z)=\frac{1}{f_\star\nu(B(z,r))}\int_{f^{-1}B(z,r)}\inf\left\{t>0, \Psi_ty\notin f^{-1}B(z,r)\right\}d \nu(y).$$
\item $$\rho_r(z)=\frac{|\log r|}{f_\star\nu(B(z,r))}\int_{f^{-1}B(z,r)}\inf\left\{t>0, \Psi_ty\notin f^{-1}B(z,r)\right\}d \nu(y).$$
\item $$\rho_r(z):=\underset{y\in f^{-1}B(z,r)}{\textrm{ess-sup}}\left\{\inf\left\{t>0:\Psi_ty\notin f^{-1}B(z,r)\right\}\right\}.$$
\end{enumerate}
We introduce now the notion of return time and recurrence rates for the observation of the flow:
\begin{definition}Let $f:M\rightarrow \R^N$ a measurable function and $\rho$ an escape function, we define for $x\in M$ the return time for the observation of the flow $\Psi$ with respect to $\rho$:
\[\tau_{r,\rho}^{\Psi,f}(x):=\inf\{t>\rho_r(f(x))\,:\,f(\Psi_t(x))\in B(f(x),r)\}\]
where $B(f(x),r)$ denotes the ball centered in $f(x)$ and of radius $r$. We also define the lower and upper recurrence rates for the observation of the flow $\Psi$ with respect to $\rho$:
\[\underline{R}^{\Psi,f}_\rho(x):=\liminf_{r\rightarrow0}\frac{\log\tau_{r,\rho}^{\Psi,f}(x)}{-\log r}\qquad\textrm{and}\qquad\overline{R}^{\Psi,f}_\rho(x):=\limsup_{r\rightarrow0}\frac{\log\tau_{r,\rho}^{\Psi,f}(x)}{-\log r}.\]
Let $p\in \N$. We also introduce the $p$-non-instantaneous return time for the observation of the flot $\Psi$:
\[\tau_{r,p}^{\Psi,f,\star}(x)=\inf\{t>p\,:\,f(\Psi_t(x))\in B(f(x),r)\}\]
and the lower and upper non-instantaneous recurrence rates for the observation of the flow $\Psi$:
\[\underline{R}_\star^{\Psi,f}(x)=\lim_{p\rightarrow+\infty}\liminf_{r\rightarrow0}\frac{\log\tau_{r,p}^{\Psi,f,\star}(x)}{-\log r}\]
and\[\overline{R}^{\Psi,f}_\star(x)=\lim_{p\rightarrow+\infty}\limsup_{r\rightarrow0}\frac{\log\tau_{r,p}^{\Psi,f,\star}(x)}{-\log r}.\]
\end{definition}
As it is explain in \cite{RFRDS}, it can be useful to work with non-instantaneous return time since, with the observation, there can be a loss of information.
We can notice that these recurrence rates can be equal:
\begin{proposition}\label{artoosftheseremreg} Let $\rho$ an escape function. If for $x\in M$, $\underline{R}^{\Psi,f}_{\rho}(x)>0$ then for any other escape function $\tilde\rho$ such that $\tilde\rho_r(f(x))\geq \rho_r(f(x))$ for all $r$ small enough and such that $\underline\lim_{r\rightarrow0}\frac{\log \tilde\rho_r(f(x))}{\log r}\geq0$, we have
\[\underline{R}^{\Psi,f}_\rho(x)=\underline{R}^{\Psi,f}_{\tilde\rho}(x)=\underline{R}_\star^{\Psi,f}(x)\]
and
\[\overline{R}^{\Psi,f}_\rho(x)=\overline{R}^{\Psi,f}_{\tilde\rho}(x)=\overline{R}_\star^{\Psi,f}(x).\]
\end{proposition}
\begin{proof}We just need to remark that if for $x\in M$, $\underline{R}^{\Psi,f}_\rho(x)>0$ then it exists $a>0$ such that for all $r>0$ small enough we have $$\tau_{r,\rho}^{\Psi,f}(x)\geq r^{-a}.$$
\end{proof}

For discrete dynamical systems, it is proved in \cite{MR1833809} that recurrence rates are linked to the pointwise dimensions of the invariant measure and in \cite{prfo} that recurrence rates for the observation are linked to the pointwise dimension of the push-forward measure. For flows, the recurrence rates for the observation are linked to the pointwise dimension of the push-forward measure but intrinsically since theses rates depend also on the escape function. The idea is to find an optimal escape function, i.e. an escape function such that almost everywhere the escape time of the observation of the flow from $B(f(x),r)$ is smaller than the escape function when the radius is small enough and such that $\rho_r(f(x))$ is the smallest possible when $r\rightarrow0$ . So, the first theorem on recurrence for observations of flows is the following (it will be proved in Section~\ref{artoosfthesesecmajo}):

\begin{theorem}\label{artoosftheseth1}
Let $\Psi$ a flow on $M$ and $\nu$ a probability measure $\Psi$-invariant. Let $f$ a measurable observation from $M$ to $\R^N$ and let $\rho$ an escape function. Then for $\nu$-almost every $x\in M$
\[\underline{R}^{\Psi,f}_\rho(x)\leq\underset{r\rightarrow0}{\liminf}\left(\frac{\log f_\star\nu\left(B\left(f(x),r\right)\right)}{\log r}-\frac{\log \rho_r(f(x))}{\log r}\right)\]
and
\[\overline{R}^{\Psi,f}_\rho(x)\leq\underset{r\rightarrow0}{\limsup}\left(\frac{\log f_\star\nu\left(B\left(f(x),r\right)\right)}{\log r}-\frac{\log \rho_r(f(x))}{\log r}\right).\]
\end{theorem} 
We emphasize that this result cannot be obtained using the time-1 map, i.e. $Tx=\Psi_1(x)$, which gives a non-optimal results. Indeed, using the results of \cite{prfo} with the time-1 map, we obtain that for $\nu$-almost every $x\in M$
\[\underline{R}^{\Psi,f}_\rho(x)\leq\underline{d}_{f_\star\nu}(f(x))\qquad and \qquad
\overline{R}^{\Psi,f}_\rho(x)\leq\overline{d}_{f_\star\nu}(f(x)).\]
This result is optimal when
\[\lim_{r\rightarrow0}\frac{\log \rho_r(f(x))}{\log r}=0\]
but it is a degenerate case as we will see in Section~\ref{artoosfthesesecexemple}.

We notice that when the observation is the identity, we obtain the following result, given in the introduction (Theorem~\ref{artoosftheseintroid}), for points which are not fixed points (it exists $t\in \R$ such that $\Psi_tx\neq x$):
\begin{corollary}\label{artoosfthesecorid}
Let $\Psi$ a differentiable flow on $M$ and $\nu$ an invariant probability measure for $\Psi$. For $\nu$-almost every $x\in M$ which is not a fixed point
\[\underline{R}^{\Psi}(x)\leq \underline{d}_\nu(x)-1\qquad\textrm{and}\qquad\overline{R}^{\Psi}(x)\leq \overline{d}_\nu(x)-1.\]
\end{corollary}
\begin{proof}
Using Theorem~\ref{artoosftheseth1} with $f$ the observation identity:
\begin{eqnarray*}
f:M&\longrightarrow& M\\
x&\longmapsto& x
\end{eqnarray*}
and choosing the escape function $\rho$ as follows:
\begin{eqnarray*}
\rho_r:M&\longrightarrow& \R\\
x&\longmapsto & \rho_r(x)=r|\log r|
\end{eqnarray*}
we obtain that
\[\underline{R}^{\Psi,f}_\rho(x)\leq \underline{d}_\nu(x)-1\qquad\textrm{and}\qquad\overline{R}^{\Psi,f}_\rho(x)\leq \overline{d}_\nu(x)-1.\]
We still need to prove that $\underline{R}^{\Psi}(x)\leq \underline{R}^{\Psi,f}_\rho(x)$ and that $\overline{R}^{\Psi}(x)\leq\overline{R}^{\Psi,f}_\rho(x)$. For this, we only need to prove that for almost every $x\in M$ and for $r$ small enough we have $\tau_r^\Psi(x)\leq\tau_{r,\rho}^\Psi(x)$. By the flow box theorem \cite{MR0361232}, and since the flow is differentiable, around a non-fixed point it exists a neightboorhood $U$, a constant $\gamma$ and a time $T$ such that fo all $z\in U$ and for all $0<t<T$, $d(z,\Psi_tz)\geq\gamma t$. Then, for all $x\in M$ which is not a fixed point, it exists $r_1(x)>0$ such that for all $0<r<r_1$, $\inf\{t>0, \Psi_tx\notin B(x,r)\}\leq r|\log r|$ and then  $\tau_r^\Psi(x)\leq\tau_{r,\rho}^\Psi(x)$.
\end{proof}
As in \cite{MR1833809,MR2191396,prfo}, to obtain a lower bound for recurrence rates, we need more assumptions on the system. For flows, an assumption on the speed of decay of correlations will not be optimal. Indeed, it is possible to construct Axiom A flow with an arbitrarily slow decay of correlations (e.g. \cite{MR692974}). That is the reason why we will use suspension flow and the assumption on the speed of mixing will be for the dynamic on the base.
\section{Recurrence for observations via suspension flows}\label{artoosfthesesecsusp}

From now on, we suppose that the manifold $M$ is compact.
To get a lower estimate of the recurrence rate we used the suspension or special flow.

Let $(X,\mathcal{A},\mu,d,T)$ a metric measure preserving system,  i.e. $\mathcal{A}$ is a $\sigma$-algebra, $\mu$ is a measure on $(X,\mathcal{A})$ with $\mu(X)=1$, $\mu$ is invariant for $T$ (i.e $\mu(T^{-1}A)=\mu(A)$ for all $A\in\mathcal{A}$) where $T:X\rightarrow X$ and $d$ is a metric.

 Let $\phi:X\rightarrow (0,+\infty)$ a measurable and integrable function. We define the space under the function $\phi$:
\[Y:=\left\{(u,s)\in X\times\R : 0\leq s\leq \phi(u)\right\}\]
where $(u,\phi(u))$ and $(Tu,0)$ are identified for all $u\in X$.
The {\it suspension flow} or the {\it special flow} over $T$ with height function $\phi$ is the flow $\Phi$ which acts on $Y$ by the following transformation
\[\Phi_t(u,s)=(u,s+t)\qquad\forall(u,s,t)\in X\times \R^+\times\R^+.\]
The metric on $Y$ is the Bowen-Walters distance (e.g. \cite{MR0341451}). We recall that the measure $\nu_\mu$ defined by $\mu\otimes Leb$ on $Y$ and normalized is invariant for the flow $\Phi$. Generally, flow and suspension flow are linked:
\begin{theorem}[\cite{MR0005800}]
Any flow $\Psi$ without fixed points is metrically isomorphic to a special flow $\Phi$.
\end{theorem}
It means that for a flow $\Psi$ on M with invariant measure $\nu$, it exists a suspension flow $\Phi$ over $T$ with height function $\phi$ and an invariant measure $\nu_\mu$ such that: it exists a set $M'$ such that $\nu(M')=1$, a set $Y'$ such that $\nu_\mu (Y')=1$ and a function $g: Y\rightarrow M$ which is one-to-one from $Y'$ to $M'$, such that for all $t\geq0$:
\[\Psi_t \circ g=g\circ \Phi_t\qquad \textrm{and}\qquad \Psi_t \circ g^{-1}=g^{-1}\circ\Phi_t\]
and such that for all measurable subset $A\subset M$ and all measurable subset $B\subset Y$:
\[\nu(A)=\nu_\mu(g^{-1}A)\qquad \textrm{and}\qquad \nu_\mu (B)=\nu(g(B)).\]
We emphasize that for hyperbolic flows, we can choose the metric on $X$ such that the function $g$ is Lipschitz \cite{MR0339281}.

In the study of the recurrence for observations of suspension flows, it appears that the recurrence rates are bounded from below by some new quantities, that is the reason why we introduce the definition of the local projection dimension of the associated suspension flow:
\begin{definition}Let $\Psi$ a flow metrically isomorphic to a suspension flow $\Phi$ over $(X,\mathcal{A},T,\mu)$ with height function $\phi$. Let $g:Y\rightarrow M$ the isomorphism linking $\Psi$ and $\Phi$. We define the lower and upper local projection dimension of the associated suspension flow for the observation $f$ at a point $x\in M$
\[\underline{d}^{f,g}_\mu(x)=\underset{r\rightarrow0}{\underline\lim}\frac{\log \mu\left(\pi (f\circ g)^{-1}B(f(x),r)\right)}{\log r}\] and \[ \overline{d}^{f,g}_\mu(x)=\underset{r\rightarrow0}{\overline{\lim}}\frac{\log \mu\left(\pi (f\circ g)^{-1}B(f(x),r)\right)}{\log r}\]
where $\pi:Y\rightarrow X$ is the projection on $X$, i.e. $\forall(u,t)\in Y$, $\pi(u,t)=u$.
\end{definition}

Even if these dimensions are not identical to the local dimensions of the push-forward measure, in some case we can link them (see Section~\ref{artoosfthesesecgeo} on the geodesic flow for example).
These dimensions appear naturally in our study. Indeed, when $f=id$, i.e when there is no observation, the return time of a point $(x,s)\in A\times B\subset Y$ in the set $AxB$ under the action  of the suspension flow is bounded from below by the return time of $x$ in $A$ under the action of $T$ and bounded from above by this same return time plus a constant (depending only on the height function $\phi$). In \cite{MR1833809} this idea was already used to prove that the asymptotic behaviour of the return time for a suspension flow of a point $(x,s)$ in $B((x,s),r)$ was equal to the asymptotic behaviour of the return time of $x$ in $B(x,r)$ under the action of $T$ which is linked to the asymptotic behaviour of the measure of $B(x,r)$. When there is an observation $f$, the return time of $f(x,s)$ in a set $C$ correspond to the return time of $(x,s)$ in $f^{-1}C$ under the action of the suspension flow. So, it is natural to try to link the return time of $(x,s)$ in $f^{-1}C$ to the measure of the projection on $X$ of the set $f^{-1}C$. For more details see Lemma \ref{artoosfthesepositive} and \ref{artoosftheseintervalle}.

 Moreover, we recall the definition of the decay of correlations:
\begin{definition}
$(X,T,\mu)$ has a super-polynomial decay of correlations if, for all $\varphi_1$, $\varphi_2$ Lipschitz functions from $X$ to $\R$ and for all $n\in\N^*$,  we have:
\[\left|\int_X\varphi_1\circ T^n\,\varphi_2 d\mu-\int_X \varphi_1 d\mu\int_X\varphi_2 d \mu\right|\leq\|\varphi_1\|\|\varphi_2\|\theta_n\]
with $\lim_{n\rightarrow\infty}n^k\theta_n =0$ for all $k>0$ and where $\|.\|$ is the Lipschitz norm.
\end{definition}
The second main theorem is the following:
\begin{theorem}\label{artoosfthesethrd}
Let $\Psi$ a flow from $M$ to $M$ and $\nu$ an invariant probability measure for $\Psi$. Let $f$ a Lipschitz observation from $M$ to $\R^N$. If $\Psi$ is metrically isomorphic to a suspension flow $\Phi$ over $T$ such that $T$ has a super-polynomial decay of correlations and if the isomorphism $g$ linking these two flows is Lipschitz, then
\[\underline{R}_\star^{\Psi,f}(x)\geq \underline{d}^{f,g}_\mu(x)\qquad and\qquad \overline{R}_\star^{\Psi,f}(x)\geq\overline{d}^{f,g}_\mu(x)\]
for $\nu$-almost every $x\in M$ such that $\underline{d}^{f,g}_\mu(x)>0$.
\end{theorem}
\begin{remark}We notice that the previous assumptions on the flow $\Psi$ are satisfied, for example, when the flow $\Psi$ is an Anosov flow and when the measure $\mu$ is an equilibrium state of an H\"older potential.
\end{remark} 
We recall the definition of an Anosov flow:
\begin{definition} A differentiable flow $\Psi_t:M\rightarrow M$ is called an Anosov flow (or uniformly hyperbolic) if it exists a constant $0<\lambda<1$, a constant $C>0$ and a decomposition of the tangent bundle:
\[TM=E^u\oplus E^0\oplus E^s\qquad(i.e.\, \forall x\in M, T_xM=E^s_x+E^0_x+E^u_x)\]
where $E^0_x$ is generated by $\frac{d}{dt}(\Psi_tx)\vert_{t=0}$, $E_u$ and $E_s$ are sub-bundles $D\psi_t$-invariant for all $t\in\R$ (i.e. $D_x\Psi_tE^u_x=E^u_{\Psi_tx}$ and $D_x\Psi_tE^s_x=E^s_{\Psi_tx}$ for all $t\in\R$) such that for all $x\in M$ and for all $t>0$
\[\|d_x\Psi_tv\|\leq C\lambda^t\|v\|\textrm{ for all $v\in E^s_x$}\]
and
\[\|d_x\Psi_{-t}v\|\leq C\lambda^t\|v\|\textrm{ for all $v\in E^u_x$}.\]
\end{definition}
When the observation is the identity map, we obtain an equality between recurrence rate and dimension (and for Anosov flows we obtain the Theorem~\ref{artoosftheseintroideg}):
\begin{corollary}\label{artoosfthesecoroid}
Let $\Psi$ a differentiable flow of $M$ and $\nu$ an invariant probability measure for $\Psi$. If $\Psi$ is metrically isomorphic to a suspension flow $\Phi$ over $T$ such that $T$ has a super-polynomial decay of correlations and if the isomorphism $g$ linking these two flows is Lipschitz, then
\[\underline{R}^{\Psi}(x)= \underline{d}_\nu(x)-1\qquad and\qquad \overline{R}^{\Psi}(x)= \overline{d}_\nu(x)-1\]
for $\nu$-almost every $x\in M$ not periodic and such that $\underline{d}_\nu(x)>1$.
\end{corollary} 
For $C^1$ flow with strictly positive entropy (for flows, the entropy correspond to the entropy of the time-1 map, i.e. $h_\nu(\Psi)=h_\nu(\Psi_1)$), the lower local dimension satisfies another condition:
\begin{lemma}\label{artoosftheselemdun}
Let $\Psi$ a flow $C^1$ on $M$ and let $\nu$ an invariant probability measure such that $h_\nu(\Psi)>0$. Then, for $\nu$-almost every $x\in M$
\[\underline{d}_\nu(x)>1.\]
\end{lemma}
\begin{proof}
Let $\eps>0$. For every $x\in M$, we consider a submanifold $N_x$ of dimension $\dim M-1$ and transverse to the flow $\Psi$. We consider now the family of balls $\{D_\delta(x)\}_{\eps/2\leq\delta\leq\eps}$ of $N_x$, of center $x$ and diameter $\delta$. Using the construction from Lemma 5 of  \cite{MR1833809}, we can prove that it exists $\delta_0\in[\eps/2,\eps]$ such that $D_{\delta_0}(x)$ satisfies for every $r>0$ small enough:
\begin{equation}\label{artoosftheseeqmud}
\mu\left(\left\{y\in N_x:d(y,\partial D_{\delta_0}(x))<r\right\}\right)\leq cr
\end{equation}
where $c$ is a constant depending only on $\eps$ and $x$ and where $\mu$ is the induced measure by $\nu$ on $D_\eps(x)$. We denoted $D(x)$ the set $D_{\delta_0}(x)$ and we consider the cylinders $C(x)=\bigcup_{0\leq t<\eps}\Psi_tD(x)$. Since the cylinders $C(x)$ cover $M$, we can extract a finite subcovering $\bigcup_{i\in I}C_i$. Let $Z=\bigcup_{i\in I} D_i$ where $D_i$ is the disk associated to the cylinder $C_i$ for $i\in I$.

We define the transfer function $\varsigma: Z\rightarrow\R^+$ by
\[\varsigma(x)=\min\{t>0:\Psi_tx\in Z\},\]
and the transfer map $T:Z\rightarrow Z$ by
\[Tx=\Psi_{\varsigma(x)}x.\]
We remark that if $D_i\cap T^{-1}D_j\neq\emptyset$ then $T\vert_{D_i\cap T^{-1}D_j}$ is $L$-Lipschitz ($\Psi_t$ being $C^1$ on $M$). Since $\mu$ is the invariant measure for $T$ induced by $\nu$, and since $h_\nu(\Psi)>0$ we have $h_\mu(T)>0$.

Let $\xi$ a partition of $Z$ finer than $\{D_i\cap T^{-1}D_j\}_{i,j\in I}$, of diameter arbitrary small and satisfying for every $0<r<\diam \xi$
\begin{equation}\label{artoosftheseeqmufront}
\mu(x\in Z:d(x,\partial \xi)<r)\leq c_1r
\end{equation}
where $c_1$ is a positive constant depending only on the diameter of $\xi$ (the construction of the partition comes from \cite{MR1833809} and \eqref{artoosftheseeqmud}). Let us consider the set $A_n=\{x\in Z:d(T^nx,\partial\xi)<e^{-n}\}$, by \eqref{artoosftheseeqmufront} we have that for every $n$ large enough $\mu(A_n)\leq c_1 e^{-n}$. Then $\sum_{n\in\N}\mu(A_n)<+\infty$ and by the Borel-Cantelli lemma , for $\mu$-almost every $x$ in $Z$, it exists $c(x)>0$ such that for all $n\in\N$, $d(T^nx,\partial\xi)>c(x)e^{-n}$. This gives us that for $\mu$-almost every $x\in Z$ and all $n\in\N$
\[B_Z(x,c(x)L^{-n}e^{-n})\subset \xi_n(x).\]
Finally, we observe that for every $x\in Z$, $t$ and $r$ small enough, denoting $y=\Psi_t x$ we have
\[B_M(y,r)\subset\bigcup_{|s|<r}\Psi_{t+s}(B_Z(x,L_1r))\]
where $L_1=\sup_{0\leq s<r}\{\textrm{Lipschitz constant of $\Psi_{-s}$}\}$. Then, for $\mu$-almost every $x$, for all $t$ small enough and $y=\Psi_tx$, denoting $r_n=c(x)L_1^{-1}L^{-n}e^{-n}$ for $n$ large enough we have
\begin{eqnarray*}
\nu(B_M(y,r_n))&\leq&2r_n\mu(B_Z(x,c(x)L^{-n}e^{-n}))\\
&\leq&2r_n\mu(\xi_n(x)).
\end{eqnarray*}
If the diameter of $\xi$ is small enough then $h_\mu(T,\xi)>\frac{h_\mu(T)}{2}>0$ and for $n$ large enough $\mu(\xi_n)\leq e^{-nh_\mu(T)/2}$ and then \[\underline{d}_\nu(y)\geq 1+\frac{h_\mu(T)}{2(1+\log L)}>1.\]
Finally, choosing  $\eps$ arbitrary small, we obtain that for $\nu$-almost every $y\in M$
\[\underline{d}_\nu(y)>1.\]
\end{proof}
\begin{remark}Using finer inequalities than in the previous proof, it seems that we will obtain that for $\nu$-almost every $x\in M$
\[\underline{d}_\nu(x)\geq1+ h_\nu(\Psi)\left(\frac{1}{\Lambda_u(x)}-\frac{1}{\Lambda_s(x)}\right)\]
where $\Lambda_u(x)$ is the greatest Lyapunov exponent of the flow and where $\Lambda_s(x)$ the smallest Lyapunov exponent.
\end{remark}
We are now ready to prove Theorem~\ref{artoosftheseintroideg}:
\begin{proof}[Proof of Theorem~\ref{artoosftheseintroideg}] This theorem comes from Corollary~\ref{artoosfthesecoroid} and Lemma~\ref{artoosftheselemdun}. Indeed, an Anosov flow is metrically isomorphic to an hyperbolic symbolic flow (a suspension flow whose base is a subshift of finite type) \cite{MR0339281}. Since $\nu$ is an equilibrium state of an H\"older potential, $\mu$ is also an equilibrium state for the dynamical system $(X,\mathcal{A},T)$, then the decay of correlations is super-polynomial and the entropy is strictly positive. We notice that the metric on $X$ can be chosen such that the isomorphism linking these two flows is Lipschitz \cite{MR0339281}.

Finally, for an Anosov flow, the number of periodic orbits is enumerable and since the measure is an equilibrium state, the set of periodic points has zero measure. 
\end{proof}
We emphasize that this result is an extension of the result for suspension flow of \cite{MR1833809}.
In fact, Theorem~\ref{artoosfthesethrd} comes from the following theorem for suspension flow:
\begin{theorem}\label{artoosftheseth2}
Let $(X,\mathcal{A},\mu,T)$ a dynamical system with a super-polynomial decay of correlations and $\Phi$ the suspension  flow over $T$ with height function $\phi$. Let $\tilde{f}=f\circ g$ a Lipschitz function. Then, we have
\[\underline{R}_\star^{\Phi,\tilde{f}}(y)\geq\underline{d}^{f,g}_\mu(g(y))\qquad and\qquad \overline{R}_\star^{\Phi,\tilde{f}}(y)\geq \overline{d}^{f,g}_\mu(g(y))\]
for $\mu\otimes Leb$-almost every $y\in Y$ such that $\underline{d}^{f,g}_\mu(g(y))>0$.
\end{theorem}
From now on without loss of generality let us assume that $\int_X\phi d\mu=1$ and so the invariant measure for the suspension flow is $\nu_\mu=\mu\otimes Leb$.
\section{Observations of the geodesic flow}\label{artoosfthesesecgeo}
Let $(M,h)$ a $C^\infty$ compact Riemannian manifold, i.e. $M$ is a $C^\infty$ compact manifold and $h$ is a $C^\infty$ Riemannian metric on $M$. The metric $h$ induces a distance $d$ on $M$. We suppose that the metric space $(M,d)$ is complete. So, by Hopf-Rinow Theorem $M$ is geodesically complete. Let $T^1M$ the unit tangent bundle of $(M,h)$. For $(p,v)\in T^1M$, $\gamma_{(p,v)}:\R\rightarrow M$ denotes the geodesic of initial conditions $\gamma_{(p,v)}(0)=p$ and $\dot\gamma_{(p,v)}(0)=v$. We now define the geodesic flow $\Psi_t:T^1M\rightarrow T^1M$ given by $\Psi_t(p,v)=(\gamma_{(p,v)}(t),\dot\gamma_{(p,v)}(t))$. Let $\Pi$ the canonical projection
\begin{eqnarray*}
\Pi\,\,:T^1M&\longrightarrow& M\\
(p,v)&\longmapsto& p.
\end{eqnarray*}
We recall that for every $(p,v)\in T^1M$ if $t>0$ is small enough
\begin{equation}\label{artoosftheseeqgeooo}
d(\Pi(p,v),\Pi(\Psi_t(p,v)))=t.
\end{equation}
Let $\nu$ an invariant probability measure for the geodesic flow.
\begin{theorem}\label{artoosfthesethgeogenintroid}Let $\Psi$ the geodesic flow defined on $T^1M$. Then for $\nu$-almost every $(p,v)\in T^1M$
\[\underline{R}^{\Psi}(p,v)\leq\underline{d}_{\nu}(p,v)-1\qquad and \qquad \overline{R}^{\Psi}(p,v)\leq\overline{d}_{\nu}(p,v)-1.\]
Moreover, if $M$ has a strictly negative curvature and if $\nu$ is an equilibrium state of an H\"older potential, then
\[\underline{R}^{\Psi}(p,v)=\underline{d}_{\nu}(p,v)-1\qquad and\qquad
\overline{R}^{\Psi}=\overline{d}_{\nu}(p,v)-1\]
for $\nu$-almost every $(p,v)\in T^1M$.
\end{theorem}
\begin{proof} The first part of this theorem can be proved using the same technic of Corollary~\ref{artoosfthesecorid} and using \eqref{artoosftheseeqgeooo}.

For the second part of the theorem, since $M$ has a strictly negative curvature, $\Psi$ is an Anosov flow and is metrically isomorphic to an hyperbolic symbolic flow \cite{MR0339281}. Moreover, the metric on $X$ can be chosen such that the isomorphism linking these two flows is Lipschitz. Since $\nu$ is an equilibrium state of an H\"older potential, $\mu$ is also an equilibrium state for the dynamical system $(X,\mathcal{A},T)$, then the decay of correlations is super-polynomial. We can apply Corollary ~\ref{artoosfthesecoroid} and Lemma~\ref{artoosftheselemdun} (the entropy of the flow being strictly positive).
\end{proof}
We are now interested by the return time of the position on the manifold $M$, i.e. the return time of the projection of the flow on the manifold:
\begin{theorem}\label{artoosfthesethgeogen}For $\nu$-almost every $(p,v)\in T^1M$
\[\underline{R}^{\Psi,\Pi}(p,v)\leq\underline{d}_{\Pi_*\nu}(p)-1\qquad and \qquad \overline{R}^{\Psi,\Pi}(p,v)\leq\overline{d}_{\Pi_*\nu}(p)-1.\]
\end{theorem}
\begin{remark} This result is still valid when the manifold $M$ is not compact and when $\nu$ is a probability measure.
\end{remark}
\begin{proof}[Proof of Theorem \ref{artoosfthesethgeogen}] The result comes from Theorem \ref{artoosftheseth1} using the escape function $\rho_r(p)=r$. In fact, it is important to notice that by \eqref{artoosftheseeqgeooo} for $r$ small enough
\[\inf\{t>0:\Pi(\Psi_t(p,v))\notin B(p,r)\}=r.\]
\end{proof}
\begin{definition}
A point $(p,v)\in T^1M$ is multiple if it exists $t>0$ such that $\Pi(\Psi_t(p,v))=p$.
\end{definition}
\begin{theorem}\label{artoosfthesethgeoneg}If $M$ has a strictly negative curvature and if $\nu$ is an equilibrium state of an H\"older potential then
\[\underline{R}^{\Psi,\Pi}(p,v)=\underline{d}_{\Pi_*\nu}(p)-1\qquad and \qquad
\overline{R}^{\Psi,\Pi}(p,v)=\overline{d}_{\Pi_*\nu}(p)-1\]
for $\nu$-almost every $(p,v)\in T^1M$ not multiple such that $\underline{d}_{\Pi_*\nu}(p)>1$.
\end{theorem}
\begin{remark}
 For every point $p\in M$ it exists enumerably many $v\in T^1_pM$ such that $(p,v)$ is multiple. We can wonder if when the measure $\nu$ is an equilibrium state of an H\"older potential, the set of multiple points is a zero measure set.
\end{remark}
Using a result of Ledrappier-Lindenstrauss \cite{MR1951544} we obtain the following corollary:
\begin{corollary}If $M$ is a compact Riemannian surface with striclty negative curvature, if $\nu$ is an equilibrium state of an H\"older potential such that $\dim_H\nu>2$ then
\[R^{\Psi,\Pi}(p,v)=1\]
for $\nu$-almost every $(p,v)\in T^1M$ not multiple.
\end{corollary}
\begin{proof} We just need to use Theorem~\ref{artoosfthesethgeoneg} and Theorem~1.1. of \cite{MR1951544}.
\end{proof}
\begin{proof}[Proof of Theorem \ref{artoosfthesethgeoneg}] We already noticed is the proof of Theorem \ref{artoosfthesethgeogenintroid} that since $M$ has a strictly negative curvature, $\Psi$ is metrically isomorphic to an hyperbolic symbolic flow and that the metric on $X$ can be choosen such that the isomorphism (denoted $g$) linking these two flows is Lipschitz. Moreover, since $\nu$ is an equilibrium state of an H\"older potential, $\mu$ is also an equilibrium state for the dynamical system $(X,\mathcal{A},T)$ and the decay of correlations is super-polynomial. The end of the proof used the following lemma (which will be proved later):
\begin{lem}\label{artoosftheselempi}For every $(p,v)\in T^1M$
\[\underline{d}^{\Pi,g}_\mu(p,v)=\underline{d}_{\Pi_*\nu}(p)-1\]
and
\[\overline{d}^{\Pi,g}_\mu(p,v)=\overline{d}_{\Pi_*\nu}(p)-1.\]
\end{lem}
Since $M$ is a compact manifold and since the flow is the geodesic flow on $T^1M$, it exists $\beta>0$ such that for every $t\leq\beta$ and for every $(x,u)\in T^1M$, $d(\Pi(x,u),\Pi(\Psi_t(x,u)))=t$. Then, for every $r$ small enough, $\tau_{r}^{\Psi,\Pi}(p,v)\geq\beta$ and so for $\nu$-almost every $(p,v)\in T^1M$ which is not multiple, $\tau_{r}^{\Psi,\Pi}(p,v)\rightarrow+\infty$ when $r\rightarrow0$. Let $k\in\N^*$. For every point $(p,v)$ not multiple it exists $r(k,p,v)>0$ such that for every $0<r<r(k,p,v)$, $\tau_{r}^{\Psi,\Pi}(p,v)>k$ which implies that $\tau_{r}^{\Psi,\Pi}(p,v)=\tau_{r,k}^{\Psi,\Pi,\star}(p,v)$. We obtain
\begin{equation}\label{artoosftheseregalrstargeo}
\underline{R}^{\Psi,\Pi}(p,v)=\underline{R}^{\Psi,\Pi}_\star(p,v)\qquad \textrm{and} \qquad \overline{R}^{\Psi,\Pi}(p,v)=\overline{R}^{\Psi,\Pi}_\star(p,v).
\end{equation}
The theorem is proved using \eqref{artoosftheseregalrstargeo}, Theorem~\ref{artoosftheseth1}, Theorem~\ref{artoosftheseth2} and Lemme~\ref{artoosftheselempi}.
\end{proof}
We now prove the essential lemma used in the previous proof.
\begin{proof}[Proof of Lemma \ref{artoosftheselempi}]
Let $(p,v)\in T^1M$. Let $r>0$.
\begin{eqnarray*}
\nu\left(\Pi^{-1}B(p,r)\right)&=&\mu\otimes Leb(g^{-1}\left(\Pi^{-1}B(p,r)\right))\\
&=&\int_X\int_{0}^{\phi(x)}\bold{1}_{\Pi^{-1}B(p,r)}(g(x,t))dt\,d\mu(x)
\end{eqnarray*}
Since $\Psi$ is the geodesic flow on $T^1M$ and since $M$ is compact, it exists $\beta>0$ such that for every $t\leq\beta$ and for every $(p,u)\in T^1M$, $d(\Pi(p,u),\Pi(\Psi_t(p,u)))=t$.
Let $x\in X$ such that it exists $t\in(0,\phi(x))$ such that $g(x,t)\in \Pi^{-1}B(p,r)$, then if $r<\frac{\beta}{2}$, for every $s\in [2r,\beta]$ we have $g(x,t+s)\notin \Pi^{-1}B(p,r)$. Indeed, 
\[g(x,t+s)=g(\Phi_{s}(x,t))=\Psi_{s}(g(x,t))\]
and so
\begin{eqnarray*}
d(p, \Pi(g(x,t+s)))&\geq& d(\Pi(g(x,t)),\Pi(g(x,t+s)))-d(p,\Pi(g(x,t)))\\
&\geq& d(\Pi(g(x,t)),\Pi(\Psi_{s}g(x,t)))-r\\
&\geq& s-r\geq 2r-r=r.
\end{eqnarray*}
Then,
\begin{equation}\label{artoosftheseinegboulepi}
\nu\left(\Pi^{-1}B(p,r)\right)\leq \frac{\sup_x\phi(x)}{\beta}2r\mu\left(\pi (\Pi\circ g)^{-1}B(p,r)\right).\end{equation}
Moreover, we notice
\begin{eqnarray*}
\nu\left(\Pi^{-1}B(p,2r)\right)&=&\int_X\int_{0}^{\phi(x)}\bold{1}_{\Pi^{-1}B(p,2r)}(g(x,t))dt\,d\mu(x)\\
&\geq& \int_{\pi (\Pi\circ g)^{-1}B(p,r)}\int_{0}^{\phi(x)}\bold{1}_{\Pi^{-1}B(p,2r)}(g(x,t))dt\,d\mu(x).
\end{eqnarray*}
Let $x\in \pi (\Pi\circ g)^{-1}B(p,r)$, it exists $t\in (0,\phi(x))$ such that $(x,t)\in (\Pi\circ g)^{-1}B(p,r)$. Then, if $r$ is small enough 
\begin{eqnarray*}
d(\Pi(g(x,t+r)),p)&\leq&d(\Pi(g(x,t+r)),\Pi(g(x,t)))+d(\Pi(g(x,t)),p)\\
&\leq&d(\Pi(\Psi_r(g(x,t))),\Pi(g(x,t)))+r\\
&\leq& r+r=2r
\end{eqnarray*}
which gives
\begin{equation}\label{artoosftheseinegpiboule}
\nu\left(\Pi^{-1}B(p,2r)\right)\geq r\mu(\pi (\Pi\circ g)^{-1}B(p,r)).
\end{equation}
Using \eqref{artoosftheseinegboulepi} and \eqref{artoosftheseinegpiboule}, we obtain
\begin{eqnarray*}\underline{d}_{\Pi_*\nu}(p)&=&\underset{r\rightarrow0}{\underline\lim}\frac{\log\nu\left(\Pi^{-1}B(p,r)\right)}{\log r}\\&=&\underset{r\rightarrow0}{\underline\lim}\frac{\log\mu\left(\pi (\Pi\circ g)^{-1}B(p,r)\right)}{\log r}+1\\&=&\underline{d}^{\Pi,g}_\mu(p,v)+1\end{eqnarray*}
and
\begin{eqnarray*}\overline{d}_{\Pi_*\nu}(p)&=&\underset{r\rightarrow0}{\overline\lim}\frac{\log\nu\left(\Pi^{-1}B(p,r)\right)}{\log r}\\&=&\underset{r\rightarrow0}{\overline\lim}\frac{\log\mu\left(\pi (\Pi\circ g)^{-1}B(p,r)\right)}{\log r}+1\\&=&\overline{d}^{\Pi,g}_\mu(p,v)+1.\end{eqnarray*}
\end{proof}
\section{Observations of suspensions flows}\label{artoosfthesesecexemple}
We now give some examples of observations of a suspension flow where we can compute the different dimensions and apply our theorems. In these examples, we will notice how important is the choice of the escape function.

Let $(X,\mathcal{A},\mu,d, T)$ a metric measure preserving system with a super-polynomial decay of correlations. We consider the suspension flow $\Phi$ over $T$ with height function $1$.
\subsection{Projection on the base}
Let $f$ the observation of the projection on the base $X$, i.e.:
 \[\begin{array}{cccl}
f:&Y&\longrightarrow&X\\
 &(x,s)&\longmapsto&x.
\end{array}\]
Then, for every $(x,s)\in Y$ we have
\[\underline{d}_{f_\star(\mu\otimes Leb)}(f(x,s))=\underline{d}_\mu(x)\qquad\textrm{and}\qquad\overline{d}_{f_\star(\mu\otimes Leb)}(f(x,s))=\overline{d}_\mu(x)\]
and
\[\underline{d}_\mu^{f,id}(x,s)=\underline{d}_\mu(x)\qquad\textrm{and}\qquad\overline{d}_\mu^{f,id}(x,s)=\underline{d}_\mu(x).\]
Choosing the escape function $\rho_r=|\log r|$, and comparing $\tau_{r,\rho}^{\Phi,f}$ to $\tau_{r,p}^{\Phi,f,\star}$, Theorem~\ref{artoosftheseth1} and Theorem~\ref{artoosfthesethrd} give us that for $\mu\otimes Leb$-almost every $(x,s)\in Y$
\[\underline{R}_\rho^{\Phi,f}(x,s)=\underline{R}_\star^{\Phi,f}(x,s)=\underline{d}_\mu(x) \]
and
\[\overline{R}_\rho^{\Phi,f}(x,s)=\overline{R}_\star^{\Phi,f}(x,s)=\overline{d}_\mu(x)\]
for $\mu\otimes Leb$-almost every $(x,s)\in Y$ such that $\underline{d}_\mu(x)>0$.

We emphasize that with this observation we are only interested in the return time of a point $x\in X$ under the action of $T$ and that we obtain the results of \cite{MR2191396} for discrete dynamical systems.

We must notice that, in this case, to obtain a non-trivial result the escape function must be chosen cautiously. Indeed, if the escape function is $\rho_r(x)=r$ for every $x$ for example, we obtain immediately that $\tau_{r,\rho}^{\Phi,f}(x,s)=r$ for every $(x,s)\in Y$ and for $r$ small enough. This come from the fact that $f^{-1}B(f(x,s),r)=B(x,r)\times[0,1]$ and so the first escape time of $\Phi_t(x,s)$ of this set does not depend on $r$ but only on $(x,s)$. In this case, we cannot obtain an equality between recurrence rates and dimensions.

\subsection{Projection on the time}
Let now $f$ the projection on the time, i.e.:
 \[\begin{array}{cccl}
f:&Y&\longrightarrow&[0,1)\\
 &(x,s)&\longmapsto&s.
\end{array}\]
Then, for every $(x,s)\in Y$ we have
\[\underline{d}_{f_\star(\mu\otimes Leb)}(f(x,s))=1\qquad\textrm{and}\qquad\overline{d}_{f_\star(\mu\otimes Leb)}(f(x,s))=1.\]
Choosing the escape function $\rho_r(s)=2r$ (which corresponds for example to the essential supremum of of the escape times of $f^{-1}B(f(x,s),r)$) we obtain that for $\mu\otimes Leb$-almost every $(x,s)\in Y$
\[\underline{R}_\rho^{\Phi,f}(x,s)=\underline{R}_\star^{\Phi,f}(x,s)=\overline{R}_\rho^{\Phi,f}(x,s)=\overline{R}_\star^{\Phi,f}(x,s)=0 .\]
We can also notice that in this case we can compute the projection dimension for the observation $f$:
\[\underline{d}_\mu^{f,id}(x,s)=0\qquad\textrm{and}\qquad\overline{d}_\mu^{f,id}(x,s)=0.\]
This example and the previous example show us that for any flow and any observation we have to choose an escape function depending on their parameters. Naturally, we can wonder if we can find a unique escape function giving optimal results for every flow. We notice that for the two previous examples, we obtain optimal results with the following escape function:
\begin{equation}\label{artoosfthesefoncfuitebon}
\rho_r(f(y))=\frac{|\log r|}{f_*\nu(B(f(y),r))}\int_{f^{-1}B(f(y),r)}\inf\{t>0:f(\Phi_t(z))\notin B(f(y),r)\}d\nu
\end{equation}
where $\nu=\mu\otimes Leb$.
\subsection{Mixed projection}
Let us suppose that $X$ is a product space, i.e. it exists two spaces $X_1$ and $X_2$ such that $X=X_1\times X_2$.
Let $f$ the following observation:
 \[\begin{array}{cccl}
f:&Y&\longrightarrow&X_1\times[0,1)\\
 &(x_1,x_2,s)&\longmapsto&(x_1,s).
\end{array}\]
Then, for every $(x_1,x_2,s)\in Y$ we have
\[\underline{d}_\mu^{f,id}(x_1,x_2,s)=\underset{r\rightarrow0}{\underline{\lim}}\frac{\log\mu(B(x_1,r)\times X_2)}{\log r}\]
and
\[\overline{d}_\mu^{f,id}(x_1,x_2,s)=\underset{r\rightarrow0}{\overline{\lim}}\frac{\log\mu(B(x_1,r)\times X_2)}{\log r}.\]
Moreover, we can easily prove that
\[\underline{d}_{f_\star\mu\otimes Leb}(f(x_1,x_2,s))=\underline{d}_\mu^{f,id}(x_1,x_2,s)+1\]
and
\[\overline{d}_{f_\star\mu\otimes Leb}(f(x_1,x_2,s))=\overline{d}_\mu^{f,id}(x_1,x_2,s)+1.\]
Then, with the escape function $\rho_r(x_1,s)=r|\log r|$ (which is in fact the escape function defined by \eqref{artoosfthesefoncfuitebon}) and using Proposition~\ref{artoosftheseremreg}, Theorem~\ref{artoosftheseth1} and Theorem~\ref{artoosfthesethrd} we can prove that for $\mu\otimes Leb$-almost every $(x_1,x_2,s)\in Y$ such that $\underline{R}_\rho^{\Phi,f}(x_1,x_2,s)>0$ we have
\[\underline{R}_\rho^{\Phi,f}(x_1,x_2,s)=\underline{R}_\star^{\Phi,f}(x_1,x_2,s)=\underline{d}_\mu^{f,id}(x_1,x_2,s)\]
and
\[\overline{R}_\rho^{\Phi,f}(x_1,x_2,s)=\overline{R}_\star^{\Phi,f}(x_1,x_2,s)=\overline{d}_\mu^{f,id}(x_1,x_2,s)\]
whenever $\underline{d}_\mu^{f,id}(x_1,x_2,s)>0$.


\section{Upper bound of the recurrence rates}\label{artoosfthesesecmajo}
To prove Theorem~\ref{artoosftheseth1}, we first need the definition of weakly diametrically regular measure:
\begin{definition}
A measure $\mu$ is weakly diametrically regular on the set $Z\subset X$ if for any $\eta>1$, for $\mu$-almost every $x\in Z$ and every $\eps>0$, there exists $\delta>0$ such that if $r<\delta$ then $\mu\left(B\left(x,\eta r\right)\right)\leq\mu\left(B\left(x,r\right)\right)r^{-\eps}$.
\end{definition} 
\begin{proof}[Proof of Theorem~\ref{artoosftheseth1}] Since any probability measure is weakly diametrically regular on $\R^d$ (for any $d\in\N^*$)~\cite{MR1833809}, the measure $f_*\nu$ is weakly diametrically regular on $\R^N$. We notice that in the definition of weakly diametrically regular measure and in Definition~\ref{artoosftheseacceptable}, the functions $\delta(f(\cdot),\eps,\eta)$ and $\zeta(f(.),\eps,\xi_1,\xi_2)$ can be made measurable for every $\eps$, $\eta$, $\xi_1$ and $\xi_2$. Let us fix $\eps>0$, $\eta=4$, $\xi_1=\frac{3}{4}$ and $\xi_2=\frac{5}{4}$. We choose $\delta>0$ small enough to have
 $$\nu(M_\delta)>\nu(M)-\eps=1-\eps$$
where $M_\delta:=\left\{x\in M:\,\delta(f(x),\eps,\eta)>\delta\textrm{ and }\zeta(f(x),\eps,\xi_1,\xi_2)>\delta\right\}$. For $\delta>r>0$ we define
\[A_\eps(r):=\left\{y\in M_\delta:\tau_{3r,\alpha}^{\Psi,f}(y)\nu\left(B(f(y),r)\right)\frac{1}{\rho_{5r}(f(y))}\geq r^{-4\eps}\right\}.\]
Let $C\subset M_\delta$ such that $(f(x))_{x\in C}$ maximal $r$-separated set for $f(M_\delta)$. We have
\[\nu(A_\eps(r))\leq \sum_{x\in C}\nu\left\{y\in f^{-1}B(f(x),r):\tau_{3r,\alpha}^{\Psi,f}(y)\nu\left( f^{-1}B(f(y),r)\right)\frac{1}{\rho_{5r}(f(y))}\geq r^{-4\eps}\right\}.\]
For $y\in f^{-1}B(f(x),2r)$ we define
\[\tau_{2r,\alpha}^{\Psi,f}(y,x)=\inf\left\{t>\rho_{4r}(f(x))r^{-\eps}\,:\,f(\Psi_t y)\in B(f(x),2r)\right\}.\]
If $y\in f^{-1}B(f(x),r)$, since $\rho_{4r}(f(x))r^{-\eps}\geq \rho_{3r}(f(y))$ and since $B(f(x),2r)\subset B(f(y),3r)$, we have
\[\tau_{2r,\alpha}^{\Psi,f}(y,x)\geq\tau_{3r,\alpha}^{\Psi,f}(y).\]
Moreover 
\[\nu\left(f^{-1}B(f(y),r)\right)\leq\nu\left(f^{-1}B(f(x),2r)\right)\]
and 
\[\rho_{4r}(f(x))r^{\eps}\leq \rho_{5r}(f(y))\]
which give us
\begin{eqnarray*}
& &\nu(A_\eps(r))\\
&\leq& \sum_{x\in C}\nu\left\{y\in f^{-1}B(f(x),r):\tau_{2r,\alpha}^{\Psi,f}(y,x)\nu\left(f^{-1}B(f(x),2r)\right)\frac{1}{\rho_{4r}(f(x))r^\eps}\geq r^{-4\eps}\right\}.
\end{eqnarray*}
Then, using Markov's inequality
\begin{equation}\label{artoosfthesemarkov}
\nu(A_\eps(r))\leq \sum_{x\in C}\frac{r^{3\eps}}{\rho_{4r}(f(x))}\nu\left( f^{-1}B(f(x),2r)\right)\int_{ f^{-1}B(f(x),2r)}\tau_{2r,\alpha}^{\Psi,f}(y,x)d\nu(y).
\end{equation}
We denote $D_r$ the set $f^{-1}B(f(x),2r)$. We define the application $T_r$ from $M$ to $M$ by
\begin{eqnarray*}
T_r\,\,:M&\longrightarrow& M\\
y&\longmapsto& \Psi_r(y)
\end{eqnarray*}
and the non-instantaneous return time for $y\in D_r$
\[\tau_{D_r}^{T_r}(y):=\inf\{n>\frac{\rho_{4r}(f(x))r^{-\eps}}{r}\,:\,T^n_ry\in D_r\}.\]
Since this return time is inferior to the $\frac{\rho_{4r}(f(x))r^{-\eps}}{r}$-th return time of $y$ in $D_r$, Ka\v c Lemma gives
\[\int_{D_r}\tau_{D_r}^{T_r}(y)d\nu(y)\leq \frac{\rho_{4r}(f(x))r^{-\eps}}{r}.\]
Moreover, for $y\in D_r$ we observe that
\[\tau_{2r,\alpha}^{\Psi,f}(y,x)\leq r\tau_{D_r}^{T_r}(y)\]
and so we obtain
\begin{equation}\label{artoosfthesekac}
\int_{f^{-1}B(f(x),2r)}\tau_{2r,\alpha}^{\Psi,f}(y,x)d\nu(y)\leq \rho_{4r}(f(x))r^{-\eps}.
\end{equation}
Using \eqref{artoosfthesemarkov} and \eqref{artoosfthesekac}, we have
\[\nu(A_\eps(r))\leq\sum_{x\in C}r^{2\eps}\nu\left(f^{-1}B(f(x),2r)\right).\]
Since the measure $f_*\nu$ is weakly diametrically regular and by definition of $C$ we obtain
\begin{eqnarray*}
\nu(A_\eps(r))&\leq&\sum_{x\in C}r^{2\eps}r^{-\eps}\nu\left(f^{-1}B\left(f(x),\frac{r}{2}\right)\right)\\
&\leq&r^{\eps}.
\end{eqnarray*}
Finally, since
\[\sum_{n\in\N}\nu(A_\eps(e^{-n}))<+\infty,\]
by the Borel-Cantelli Lemma, for $\nu$-almost every $x\in M_\delta$, it exists $N(x)\in\N$ such that for every $n>N(x)$
\[\tau_{3e^{-n},\alpha}^{\Psi,f}(x)\nu\left( f^{-1}B(f(x),e^{-n})\right)\leq e^{4\eps n}\rho_{5e^{-n}}(f(x))\]
and then
\begin{equation}\label{artoosftheseeqen}\frac{\log \tau_{3e^{-n},\alpha}^{\Psi,f}(x)}{n}\leq\frac{\log\nu\left( f^{-1}B(f(x),e^{-n})\right)}{-n}-\frac{\log \rho_{5e^{-n}}(f(x))}{-n}+4\eps.\end{equation}
Since we can prove easily that for $a>0$
\[\underline{R}^{\Psi,f}_\rho(x)=\liminf_{n\rightarrow+\infty}\frac{\log\tau_{ae^{-n},\rho}^{\Psi,f}(x)}{n}\qquad\textrm{,}\qquad\overline{R}^{\Psi,f}_\rho(x)=\limsup_{n\rightarrow+\infty}\frac{\log\tau_{ae^{-n},\rho}^{\Psi,f}(x)}{n},\]
\[\underset{r\rightarrow0}{\underline\lim}\frac{\log\nu\left( f^{-1}B(f(x),r)\right)}{\log r}-\frac{\log\rho_{r}(f(x))}{\log r}=\underset{n\rightarrow+\infty}{\underline\lim}\frac{\log\nu\left( f^{-1}B(f(x),e^{-n})\right)}{-n}-\frac{\log\rho_{5e^{-n}}(f(x))}{-n},\]
\[\underset{r\rightarrow0}{\overline\lim}\frac{\log\nu\left( f^{-1}B(f(x),r)\right)}{\log r}-\frac{\log\rho_{r}(f(x))}{\log r}=\underset{n\rightarrow+\infty}{\overline\lim}\frac{\log\nu\left( f^{-1}B(f(x),e^{-n})\right)}{-n}-\frac{\log\rho_{5e^{-n}}(f(x))}{-n},\]
and since $\eps$ can be chosen arbitrarily small, the theorem is proved taking the limit superior or the limit inferior when $n\rightarrow+\infty$ in \eqref{artoosftheseeqen}.
\end{proof}

\section{Lower bound of the recurrence rates}
In the section, we are going to prove Theorem~\ref{artoosfthesethrd}, proving in a first time Theorem~\ref{artoosftheseth2}.

Since the function $\tilde{f}$ is Lipschitz, we denote by $\tilde{L}$ its Lipschitz constant. Let $a>0$, $b>0$, $\beta>0$ and $\eta>0$. Let $Y_a:=\{y\in Y\,,\,\underline{d}^{f,g}_\mu(g(y))>a\}$. We define
\begin{eqnarray*}
G_1&=&\left\{y\in Y_a:\,\forall r\leq \eta,\,\mu\left(\pi \tilde{f}^{-1}B\left(\tilde{f}(y),r\right)\right)\leq r^a\right\},\\
G_2&=&\left\{y=(x,s)\in Y_a:\,\forall r\leq \eta,\,\mu\left(B\left(x,\frac{r}{2}\right)\right)\geq r^{N+b}\right\},\\
G_3&=&\left\{y=(x,s)\in Y_a:\,\forall r\leq \eta,\,\mu\left(B\left(x,\frac{r}{2}\right)\right)\geq \mu\left(B\left(x,2r\right)\right)r^{\beta}\right\}.
\end{eqnarray*}
We notice that $G(a,b,\beta,\eta):=G_1\cap G_2\cap G_3$ satisfies
\begin{equation}\label{artoosfthesemug2}
{\mu\otimes Leb}(G(a,b,\beta,\eta))\underset{\eta\rightarrow0}{\longrightarrow}{\mu\otimes Leb}(Y_a).
\end{equation}
Indeed, by definition of $\underline{d}^{f,g}_\mu$, we have ${\mu\otimes Leb}(G_1)\rightarrow{\mu\otimes Leb}(Y_a)$. Moreover, since $\overline{d}_\mu\leq N$ $\mu$-almost everwhere, ${\mu\otimes Leb}(G_2)\rightarrow {\mu\otimes Leb}(Y_a)$, and since the measure $\mu$ is weakly diametrically regular, ${\mu\otimes Leb}(G_3)\rightarrow{\mu\otimes Leb}(Y_a)$.

Let $\alpha>0$, we define
\[Y(\alpha,a)=\{y=(x,s)\in Y_a\,:\,\alpha<s<\phi(x)-\alpha\}.\]
We consider the set $G=G(a,b,\beta,\eta)\cap Y(\alpha,a)$. Following the ideas of P\`ene and Saussol  \cite{MR2566164} for return times in billiard, we will use a special cover of $Y$. Since $X$ is compact, it exists, for $r>$ small enough, a finite subset $E=(m_i)_{i\in I}\subset X$ and a finite sequence $(s_{ij})_{(i,j)\in I\times J}\subset\R$ such that $\left\{P_{ij}(r):=\left\{\Phi_s(B(m_i,r)\times\{s_{ij}\}),0\leq s\leq r\right\}\right\}_{(i,j)\in I\times J}$ satisfies
\begin{enumerate}
\item $(m_i,s_{ij})\in G$ for every ${(i,j)\in I\times J}$;
\item $G\subset\underset{(i,j)\in I\times J}{\bigcup}P_{ij}(r)$;
\item $B(m_{i_1},\frac{r}{2})\cap B(m_{i_2},\frac{r}{2})=\emptyset$ for every $i_1\neq i_2$;
\item $s_{ij}\in(0,\phi(m_i))$ for every $i,j$;
\item $Y(\alpha,a)\subset \underset{i,j}{\bigcup}P_{ij}(r)\subset Y(\frac{\alpha}{2},a)$;
\item $P_{ij}(\frac{r}{2})\cap P_{kl}(\frac{r}{2})=\emptyset$ for every $(i,j)\neq (k,l)$.
\end{enumerate}
Let $r\leq\eta$, we define:
\[A_r^{\tilde{f}}(y)=\left\{x\in X\,:\,\exists t\in[0,\phi(x)[\,,\,\Phi_t(x,0)\in \tilde{f}^{-1}B(\tilde{f}(y),r)\right\}.\]

\begin{lem}\label{artoosfthesedoc}Under the hypothesis of Theorem \ref{artoosftheseth2}, for every $(x,s)\in Y$, for every $n\in \N$, for every $K>0$ and for every $r>0$ we have
\[\mu\left(B(x,r)\cap T^{-n}A_{Kr}^{\tilde{f}}(x,s)\right)\leq \frac{\tilde{L}c}{K}\frac{1}{r^2}\theta_n+\mu\left(B(x,2r)\right)\mu\left(A_{2Kr}^{\tilde{f}}(x,s)\right)\]
where $c$ is a strictly positive constant depending only on the different metrics.
\end{lem}

\begin{proof}
Let $(x,s)\in Y$ and $r>0$. Let ${h}_{x,r}$ and ${h'}_{(x,s),r}$ defined as follow
\[\begin{array}{cccl}
h_{x,r}:&X&\longrightarrow&\R\\
 &u&\longrightarrow &h(u)=\max\left\{0,1-\frac{1}{r}d(u,B(x,r))\right\}
 \end{array}\]
and
 \[\begin{array}{cccl}
{h'}_{(x,s),r}:&X&\longrightarrow&\R\\
 &u&\longrightarrow&\underset{t\in[0,\phi(u)[}{\sup}\max\left\{0,1-\frac{1}{Kr}d(B(\tilde{f}(x,s),Kr),\tilde{f}(u,t))\right\}.
\end{array}\]
$h_{x,r}$ is $\frac{1}{r}$-Lipschitz and ${h'}_{(x,s),r}$ is $\frac{\tilde{L}c}{Kr}$-Lipschitz with $c$ the constant given by Proposition~17 of \cite{MR1796025}. Moreover, we have $1_{B(x,r)}\leq h_{x,r}\leq 1_{B(x,2r)}$ and $1_{A_{Kr}^{\tilde{f}}(x,s)}\leq {h'}_{(x,s),r}\leq 1_{A_{2Kr}^{\tilde{f}}(x,s)}$. Since the decay of correlations of $(X,T,\mu)$ is super-polynomial,we obtain
\begin{eqnarray*}
\mu\left(B(x,r)\cap T^{-n}A_{Kr}^{\tilde{f}}(x,s)\right)&\leq&\int_X h_{x,r}(u){h'}_{(x,s),r}(T^nu)d\mu(u)\\
&\leq& \|h_{x,r}\|\|{h'}_{(x,s),r}\|\theta_n+\int_X h_{x,r}d\mu\int_X {h'}_{(x,s),r}d\mu\\
&\leq& \frac{\tilde{L}c}{K}\frac{1}{r^2}\theta_n+\mu\left(B(x,2r)\right)\mu\left(A_{2Kr}^{\tilde{f}}(x,s)\right).
\end{eqnarray*}
\end{proof}


\begin{lem}\label{artoosfthesepositive}Under the hypothesis of Theorem \ref{artoosftheseth2},\[\underline{R}_\star^{\Phi,\tilde{f}}(y)> 0\textrm{  for ${\mu\otimes Leb}$-almost every $y$ such that }\underline{d}^{f,g}_\mu(g(y))>0.\]
\end{lem}

\begin{proof}Let $Y_+:=\{\underline{d}^{f,g}_\mu(g(y))>0\}$. Let $1>\eps>0$ and $a>0$ such that ${\mu\otimes Leb}(Y_+)\geq{\mu\otimes Leb}(Y_a)>{\mu\otimes Leb}(Y_+)-\eps$. Let $\alpha>0$. We fix $b>0$, $\beta=\frac{a}{2}$ and for $\eta>0$ we consider the set $G=G(a,b,\beta,\eta)\cap Y(\alpha,a)$ defined previously.
Let $n_0\in\N$ such that $\forall n\geq n_0$, we have $\eps_n=\frac{1}{n^{4/a}}<\eta$ and we define
\[H_n:=\{y=(x,s)\in Y(\alpha,a)\,:\,T^nx\in A_{r_n}^{\tilde{f}}(y)\}.\]
We consider the set $\left\{P_{ij}(r_n)\right\}_{(i,j)\in I\times J}$ defined previously. Then,
\begin{eqnarray*}
{\mu\otimes Leb}(G\cap H_n)&\leq&\sum_{(i,j)\in I\times J}{\mu\otimes Leb}\left((x,s)\in P_{ij}(r_n):T^nx\in A_{r_n}^{\tilde{f}}(x,s)\right)\\
&\leq&\sum_{(i,j)\in I\times J}{\mu\otimes Leb}\left((x,s)\in P_{ij}(r_n):T^nx\in A_{(1+2\tilde{L}c)r_n}^{\tilde{f}}(m_i,s_{ij})\right).
\end{eqnarray*}
The definition of $P_{ij}(r_n)$ gives us for every ${(i,j)\in I\times J}$
\begin{eqnarray}
& &{\mu\otimes Leb}\left((x,s)\in P_{ij}(r_n):T^nx\in A_{(1+2\tilde{L}c)r_n}^{\tilde{f}}(m_i,s_{ij})\right)\nonumber \\
&=&{\mu\otimes Leb}\left((x,s)\in B(m_i,r_n)\times[s_{ij},s_{ij}+r_n]:T^nx\in A_{(1+2\tilde{L}c)r_n}^{\tilde{f}}(m_i,s_{ij})\right)\nonumber\\
 &=&r_n\mu\left(x\in B(m_i,r_n):T^nx\in A_{(1+2\tilde{L}c)r_n}^{\tilde{f}}(m_i,s_{ij})\right).\label{artoosftheseeqpij}
\end{eqnarray}
Then, by Lemma~\ref{artoosfthesedoc} and \eqref{artoosftheseeqpij}, we obtain
\begin{eqnarray*}
& &{\mu \otimes Leb}(G\cap H_n)\\
&\leq&\sum_{(i,j)\in I\times J}r_n\left[\frac{\tilde{L}c}{2\tilde{L}c+1}\frac{1}{r_n^2}\theta_n+\mu\left(B(m_i,2r_n)\right)\mu\left(A_{2(1+2\tilde{L}c)r_n}^{\tilde{f}}(m_i,s_{ij})\right)\right].
\end{eqnarray*}
By definition of $G$ we have
\begin{eqnarray*}
& &{\mu\otimes Leb}(G\cap H_n)\\
&\leq&\sum_{(i,j)\in I\times J}\left[\frac{\tilde{L}c}{2\tilde{L}c+1}r_n^{-N-b-1}\theta_n+r_n^{1-\beta}(2(1+2\tilde{L}c)r_n)^a\right]\mu\left(B(m_i,\frac{r_n}{2})\right)
\end{eqnarray*}
and by definition of $P_{ij}$
\[{\mu\otimes Leb}(G\cap H_n)\leq\|\phi\|\left[\frac{\tilde{L}c}{2\tilde{L}c+1}r_n^{-N-b-2}\theta_n+r_n^{-\beta}(2(1+2\tilde{L}c)r_n)^a\right].\]
Since $\sum_{n\in\N^*}r_n^{a-\beta}=\sum_{n\in\N^*}\frac{1}{n^2}<+\infty$ and since the decay of correlations is super-polynomial, we have
\[\sum_{n\in\N^*}{\mu\otimes Leb}(G\cap H_n)<+\infty.\]
By the Borel-Cantelli lemma and using (\ref{artoosfthesemug2}), we have for ${\mu\otimes Leb}$-almost every $y=(x,s)\in Y(\alpha,a)$, it exists $n_1(y)$ such that for every $n\geq n_1(y)$, $T^nx\notin A_{\frac{1}{n^{4/a}}}^{\tilde{f}}$, i.e. for every $n\geq n_1(y)$ and for every $t\in[0,\phi(T^nx)[$, $\tilde{f}(T^nx,t)\notin B(\tilde{f}(x,s),\frac{1}{n^{4/a}})$. Then, for ${\mu\otimes Leb}$-almost every $y\in Y(\alpha,a)$, for every $p\geq n_1(y)$ and for every $n\geq n_1(y)$
\begin{equation}
\tau_{\frac{1}{n^{4/a}},p}^{\Phi,\tilde{f},\star}(x,s)>n
\end{equation}
which gives
\begin{eqnarray*}
\underline{R}_\star^{\Phi,\tilde{f}}(y)&=&\lim_{p\rightarrow+\infty}\liminf_{r\rightarrow0}\frac{\log\tau_{r,p}^{\Phi,\tilde{f},\star}(y)}{-\log r}\\
&=&\lim_{p\rightarrow+\infty}\liminf_{n\rightarrow+\infty}\frac{\log\tau_{\frac{1}{n^{4/a}},p}^{\Phi,\tilde{f},\star}(y)}{-\log \frac{1}{n^{4/a}}}\\
&\geq&\lim_{p\rightarrow+\infty}\lim_{n\rightarrow+\infty}\frac{\log n}{\log n^{4/a}}=\frac{a}{4}>0.
\end{eqnarray*}
Since we can choose $\alpha$ and $\eps$ arbitrarily small, the lemma is proved.
\end{proof}

\begin{lem}\label{artoosftheseintervalle} Let $a>0$, $\delta>0$ and $1>\eps>0$. For ${\mu\otimes Leb}$-almost every $y\in Y_a:=\{ \underline{d}^{f,g}_\mu(g(y))>a\}$, it exists $r(y)>0$ such that for every $r \in]0,r(y)[$ and for every $t\in \left[r^{-\delta}, \mu\left(\pi \tilde{f}^{-1}B(\tilde{f}(y),(2\tilde{L}c+1)er)\right)^{-1+\eps}\right]$, we have $d(\tilde{f}(\Phi_ty),\tilde{f}(y))\geq r$.
\end{lem}

\begin{proof}
Let $\alpha>0$. Let $1>\eps>0$. We fix $b>0$, $\beta=\frac{a\eps}{2}$ and for $\eta>0$ we consider the set $G=G(a,b,\beta,\eta)\cap Y(\alpha,a)$. Let $\delta>0$ and $r\leq\eta$, we define:
\begin{eqnarray*}
C_\eps(r):=\bigg\{w\in Y&:& \exists t\in\left[r^{-\delta},\mu\left(\pi \tilde{f}^{-1}B(\tilde{f}(w),(6\tilde{L}c+2)r)\right)^{-1+\eps}\right]\\ 
 & & \textrm{such that }d(\tilde{f}(\Phi_tw),\tilde{f}(w))<r\bigg\}.
 \end{eqnarray*}
For $x\in G$ and $s>0$, we define $B_{r,s}(x):=B(x,r)\times[s,s+r]$ and defining $I_r=\left[\lfloor r^{-\delta}\rfloor,\mu\left(\pi \tilde{f}^{-1}B(\tilde{f}(x,s),(4\tilde{L}c+2)r)\right)^{-1+\eps}\right]\cap\N$ we can prove that we have:
\begin{equation}\label{artoosftheseeqcomplique}
B_{r,s}(x)\cap C_\eps(r)\subset\underset{I_r}{\bigcup}B(x,r)\cap T^{-n}A_{(2\tilde{L}c+1)r}^{\tilde{f}}(x,s)\times[s,s+r].
\end{equation}
Indeed, let us choose an element $w$ of $B_{r,s}(x)\cap C_\eps(r)$. By definitition, this means that $w\in B_{r,s}(x)$ and that it exists $t\in[r^{-\delta},\mu(\{\pi \tilde{f}^{-1}B(\tilde{f}(w),(6\tilde{L}c+2)r)\})^{-1+\eps}]$ such that $d(\tilde{f}(\Phi_tw),\tilde{f}(w))<r$. 

But noticing that $B(\tilde{f}(x,s),(4\tilde{L}c+2)r)\subset B(\tilde{f}(w),(6\tilde{L}c+2)r)$, the previous time $t$ is smaller than $\mu(\{\pi \tilde{f}^{-1}B(\tilde{f}(x,s),(4\tilde{L}c+2)r)\})^{-1+\eps}$. 

Now, since we work with a suspension flow over an application $T$, we denote $w=(v,u)$ and we can divide our interval in which $t$ exists in time intervals of length $1$ and with integer bounds.

Then, if $(v,u)\in B_{r,s}(x)$ then it exists an integer $n\in I_r$ and it exists a time $t\in[0,\phi(T^nv)[$ such that we have $d(\tilde{f}(T^nv,t)),\tilde{f}(v,u))<r$. 

Noticing that $\tilde{f}$ is Lipschitz and that $(v,u)\in B_{r,s}(x)$, we can see that for the previous $n$ and $t$ we also have $d(\tilde{f}(T^nv,t)),\tilde{f}(x,s))<(2\tilde{L}c+1)r$. This implies that $v\in T^{-n}A_{(2\tilde{L}c+1)r}^{\tilde{f}}(x,s)$ and so we obtain equation \eqref{artoosftheseeqcomplique}.

Let $k>1$ such that $\delta(k-1)-1\geq N+2b$ and let $\eta>0$ such that $n\geq\eta^{-\delta}$implies $(k-1)(n+1)^{-k}\geq\theta_n$ (which is possible by definition of $\theta_n$). We have by Lemma~\ref{artoosfthesedoc}
\begin{eqnarray*}
& &{\mu\otimes Leb}\left(B_{r,s}(x)\cap C_\eps(r)\right)\\
&\leq&r\underset{n\in I_r}{\sum}\left[\frac{\tilde{L}c}{(2\tilde{L}c+1)r^2}\theta_n+\mu\left(B(x,2r)\right)\mu\left(A_{2(2\tilde{L}c+1)r}^{\tilde{f}}(x,s)\right)\right]\\
&\leq&\frac{\tilde{L}c}{(2\tilde{L}c+1)}r^{\delta(k-1)-1}+r\mu(B(x),2r))\mu\left(A_{2(2\tilde{L}c+1)r}^{\tilde{f}}(x,s)\right)^{\eps}
\end{eqnarray*}
and by definition of $G$
\begin{eqnarray*}
{\mu\otimes Leb}\left(B_{r,s}(x)\cap C_\eps(r)\right)&\leq&\frac{\tilde{L}c}{(2\tilde{L}c+1)} r^{N+2b}+(4\tilde{L}c+2)^{a\eps}r^{1+\frac{a\eps}{2}}\mu(B(x,\frac{r}{2}))\\
&\leq& \mu(B(x,\frac{r}{2}))\left(\frac{\tilde{L}c}{2\tilde{L}c+1}r^b+(4\tilde{L}c+2)^{a\eps}r^{1+\frac{a\eps}{2}}\right).
\end{eqnarray*}
We consider the set $\left\{P_{ij}(r)\right\}_{(i,j)\in I\times J}$ defined previously. We have
\begin{eqnarray*}
{\mu\otimes Leb}(G\cap C_\eps(r))&\leq&\sum_{(i,j)\in I\times J}{\mu\otimes Leb}(C_\eps(r)\cap P_{ij}(r))\\
&\leq&\sum_{(i,j)\in I\times J}\mu(B(m_{i},\frac{r}{2}))\left(\frac{\tilde{L}c}{2\tilde{L}c+1}r^b+(4\tilde{L}c+2)^{a\eps}r^{1+\frac{a\eps}{2}}\right)\\
&\leq&\sum_{j\in J}\left(\frac{\tilde{L}c}{2\tilde{L}c+1}r^b+(4\tilde{L}c+2)^{a\eps}r^{1+\frac{a\eps}{2}}\right)\\
&\leq&\frac{\|\phi\|}{r}\left(\frac{\tilde{L}c}{2\tilde{L}c+1}r^b+(4\tilde{L}c+2)^{a\eps}r^{1+\frac{a\eps}{2}}\right)\\
&\leq&\frac{\tilde{L}c}{2\tilde{L}c+1}r^{b-1}+(4\tilde{L}c+2)^{a\eps}r^{\frac{a\eps}{2}}.
\end{eqnarray*}
Then, choosing $b=1$, we obtain
\[\sum_{k\in\N}{\mu\otimes Leb}(G\cap C_\eps(e^{-k}))<+\infty.\]
And so, by the Borel-Cantelli lemma, for ${\mu\otimes Leb}$-almost every $y\in G$, it exists $n_1(y)$ such that for every $k\geq n_1(y)$, $y\notin C_\eps(e^{-k})$. Then, for $r$ small enough, it exists $k\in\N$ such that $e^{-k-1}<r\leq e^{-k}\leq e^{-n_1(y)}$. Moreover, since $e^{\delta k}\leq r^{-\delta}$ and $(2\tilde{L}c+1)e^{-k}<(2\tilde{L}c+1)er$, it does not exist integer $t\in[r^{-\delta},\mu\left(\pi \tilde{f}^{-1}B(\tilde{f}(y),(2\tilde{L}c+1)er)\right)^{-1+\eps}]$ such that $d(\tilde{f}(\Phi_ty),\tilde{f}(y))\geq r$. The lemma is proved choosing $\alpha$ arbitrarily small. \end{proof}
%
%
%
%
\begin{proof}[Proof of Theorem~\ref{artoosftheseth2}]Let $\zeta>0$. Since $ \underline{R}_\star^{\Phi,\tilde{f}}(y)>0$ for ${\mu\otimes Leb}$-almost every $y\in Y_+=\{\underline{d}^{f,g}_\mu(g(y))>0\}$ by Lemma~\ref{artoosfthesepositive}, it exists $a>0$ such that ${\mu\otimes Leb}(Y_+)\geq{\mu\otimes Leb}(\{ \underline{R}_\star^{\Phi,\tilde{f}}(y)>a\})>{\mu\otimes Leb}(Y_+)-\zeta$. For every $y\in\{\underline{R}_\star^{\Phi,\tilde{f}}(y)>a\}$ , for $p$ large enough and for $r$ small enough, we have
\[\tau_{r,p}^{\Phi,\tilde{f},\star}(y)\geq r^{-a}.\]
Thanks to Lemma~\ref{artoosftheseintervalle} choosing $\delta=a$ and $\eps>0$, for ${\mu\otimes Leb}$-almost every $y\in \{\underline{R}_\star^{\Phi,\tilde{f}}(y)>a\}$, if $r$ is small enough and $p$ is large enough, then $\tau_{r,p}^{\Phi,\tilde{f},\star}(y)\geq \mu\left(\pi \tilde{f}^{-1}B(\tilde{f}(y),(2\tilde{L}c+1)er)\right)^{-1+\eps}$. Then, ${\mu\otimes Leb}$-almost everywhere on $\{\underline{R}_\star^{\Phi,\tilde{f}}(y)>a\}$ we have $\underline{R}_\star^{\Phi,\tilde{f}}(y)\geq(1-\eps)\underline{d}^{f,g}_\mu(g(y))$ and $\overline{R}_\star^{\Phi,\tilde{f}}(y)\geq(1-\eps)\overline{d}^{f,g}_\mu(g(y))$. The theorem is proved choosing $\eps>0$ arbitrarily small and then $\zeta>0$ arbitrarily small. \end{proof}

\begin{proof}[Proof of Theorem \ref{artoosfthesethrd}]
It exists $Y_g\subset Y$ such that $\mu\otimes Leb(Y_g)=\nu (g(Y_g))=1$ and such that $g$ is one-to-one on $Y_g$. Then for $\nu$-almost every $x\in M$, it exists $y\in Y_g$ such that $x=g(y)$ and noticing that
\[\underline{R}_\star^{\Psi,f}(x)=\underline{R}_\star^{\Phi,\tilde{f}}(y)\]
and that
 \[\overline{R}_\star^{\Psi,f}(x)=\overline{R}_\star^{\Phi,\tilde{f}}(y),\]
the theorem is proved using Theorem \ref{artoosftheseth2}.

\end{proof}


\begin{proof}[Proof of Corollary \ref{artoosfthesecoroid}] Let $x\in M$ a non-fixed point. We have already seen in the proof of Corollary~\ref{artoosfthesecorid}, that by the flow box theorem, it exists a neightboorhood $U$ of $x$, $\beta>0$, $\gamma_1>0$ and $\gamma_2>0$ such that for every $0<t\leq\beta$ and for every $z\in U$, $ \gamma_2 t\geq d(z,\Psi_t(z))\geq \gamma_1 t$. Let $\frac{\beta}{2}>r>0$ such that $B(x,r)\subset U$.
\begin{eqnarray*}
\nu\left(B(x,r)\right)&=&\mu\otimes Leb(g^{-1}B(x,r))\\
&=&\int_X\int_{0}^{\phi(u)}\bold{1}_{B(x,r)}(g(u,t))dt\,d\mu(u).
\end{eqnarray*}
Let $u\in X$ such that it exists $t\in(0,\phi(u))$ satisfying $g(u,t)\in B(x,r)$ then since $r<\frac{\beta}{2}$, for every $s\in [\frac{2r}{\gamma_1},\beta]$ we have $g(u,t+s)\notin B(x,r)$. Indeed, 
\[g(u,t+s)=g(\Phi_{s}(u,t))=\Psi_{s}(g(u,t))\]
and so
\begin{eqnarray*}
d(x, g(u,t+s))&\geq& d(g(u,t),g(u,t+s))-d(x,g(u,t))\\
&\geq& d(g(u,t),\Psi_{s}g(u,t))-r\\
&\geq& \gamma_1 s-r\geq 2r-r=r.
\end{eqnarray*}
Then,
\begin{equation}\label{artoosftheseinegboulepi1}
\nu\left(B(x,r)\right)\leq 2r\frac{\sup_u\phi(u)}{\gamma_1}\mu\left(\pi g^{-1}B(x,r)\right).\end{equation}
Moreover, we observe that
\begin{eqnarray*}
\nu\left(B(x,2r)\right)&=&\int_X\int_{0}^{\phi(u)}\bold{1}_{B(x,2r)}(g(u,t))dt\,d\mu(u)\\
&\geq& \int_{\pi  g^{-1}B(x,r)}\int_{0}^{\phi(u)}\bold{1}_{B(x,2r)}(g(u,t))dt\,d\mu(u).
\end{eqnarray*}
Let $u\in \pi g^{-1}B(x,r)$, it exists $t\in (0,\phi(u))$ such that $(u,t)\in g^{-1}B(x,r)$. Then, if $r$ is small enough 
\begin{eqnarray*}
d(g(u,t+\frac{r}{\gamma_2}),x)&\leq&d(g(u,t+\frac{r}{\gamma_2}),g(u,t))+d(g(u,t),x)\\
&\leq&d(\Psi_\frac{r}{\gamma_2}(g(u,t)),g(u,t))+r\\
&\leq& r+r=2r
\end{eqnarray*}
which gives
\begin{equation}\label{artoosftheseinegpiboule1}
\nu\left(B(x,2r)\right)\geq \frac{r}{\gamma_2}\mu(\pi g^{-1}B(x,r)).
\end{equation}
Using \eqref{artoosftheseinegboulepi1} and\eqref{artoosftheseinegpiboule1}, we obtain
\begin{equation}\label{artoosfthesednuegaldplusun}
\underline{d}_{\nu}(x)=\underset{r\rightarrow0}{\underline\lim}\frac{\log\nu\left(B(x,r)\right)}{\log r}=\underset{r\rightarrow0}{\underline\lim}\frac{\log\mu\left(\pi g^{-1}B(x,r)\right)}{\log r}+1=\underline{d}^{id,g}_\mu(x)+1
\end{equation}
and
\begin{equation}\label{artoosfthesednuegaldplusun2}
\overline{d}_{\nu}(x)=\underset{r\rightarrow0}{\overline\lim}\,\frac{\log\nu\left(B(x,r)\right)}{\log r}=\underset{r\rightarrow0}{\overline\lim}\,\frac{\log\mu\left(\pi g^{-1}B(x,r)\right)}{\log r}+1=\overline{d}^{id,g}_\mu(x)+1.
\end{equation}
Since for every $t\leq\beta$ and for every $y\in M$, $d(y,\Psi_t(y))\geq \gamma t$, then, for every $r$ small enough, $\tau_{r}^{\Psi}(x)\geq\gamma\beta$. And so, for $\nu$-almost every $x\in M$ which is not periodic, $\tau_{r}^{\Psi}(x)\rightarrow+\infty$ when $r\rightarrow0$. Let $p\in\N^*$. Then for every point $x$ non-periodic it exists $r(p,x)>0$ such that for every $0<r<r(p,x)$, $\tau_{r}^{\Psi}(x)>p$ which implies that $\tau_{r}^{\Psi}(x)=\tau_{r,p}^{\Psi,id,\star}(x)$. We obtain
\begin{equation}\label{artoosftheseregalrstar}
\underline{R}^{\Psi}(x)=\underline{R}^{\Psi,id}_\star(x)\qquad \textrm{and} \qquad \overline{R}^{\Psi}(x)=\overline{R}^{\Psi,id}_\star(x).
\end{equation}
Finally, the corollay is proved by \eqref{artoosfthesednuegaldplusun}, \eqref{artoosfthesednuegaldplusun2}, \eqref{artoosftheseregalrstar}, Corollary~\ref{artoosfthesecorid} and Theorem~\ref{artoosfthesethrd}. 
\end{proof}
\subsection*{Aknowledgements}I would like to thank Beno\^it Saussol for his help, his comments and all the fruitful discussions we had during the preparation of this work.

\bibliographystyle{siam} 
\bibliography{biblio}

\end{document}